\documentclass[12pt]{amsart}
\newtheorem{theorem}{Theorem}[section]
\newtheorem{lemma}[theorem]{Lemma}

\newtheorem{proposition}[theorem]{Proposition}
\theoremstyle{definition}
\newtheorem{definition}[theorem]{Definition}
\newtheorem{example}[theorem]{Example}

\usepackage{enumerate}
\newtheorem{remark}[theorem]{Remark}
\usepackage[margin=0.75in]{geometry}
\usepackage{calligra}
\usepackage{xcolor, soul}
\sethlcolor{yellow}
\usepackage{mathtools}

\DeclareMathOperator\Span{span}
\DeclareMathOperator\tr{tr}

\numberwithin{equation}{section}
\DeclareMathOperator*{\esssup}{ess\,sup}
\def\N{{\mathbb{N}}}
\def\h{{\mathcal{H}}}
\def\r{{\mathbb{R}^{n}}}
\def\R{{L^2(\mathbb{R}^{n})}}
\def\l{{L^2(\mathbb{R}^{2n})}}
\def\L{{L^{2}(\mathbb{T}^{n}\times \mathbb{T}^{n}, L^{2}(\mathbb{R}^{n}))}}
\def\Ll{{L^{2}(\mathbb{T}^{n}\times \mathbb{T}^{n};\left[\p,\p\right])}}
\def\Z{{\mathbb{Z}^n}}
\def\I{{\int\limits_{\T}}}
\def\T{{\mathbb{T}^{n}}}
\def\A{{\mathcal{A}}}
\def\J{{\j^{-1}}}
\def\j{{\mathcal{J}}}

\def\d{{\mathcal{D}}}

\def\z{{\mathbb{Z}}}
\def\t{{T_{(k,l)}^t}}
\def\c{{\xi}}
\def\d{{\xi^{'}}}
\def\e{{\eta}}
\def\p{{\phi}}
\def\v{{V^{t}}}

\newcommand\numberthis{\addtocounter{equation}{1}\tag{\theequation}}

\newcommand{\be}{\begin{equation}}
	\newcommand{\ee}{\end{equation}}

\begin{document}
	\title{Twisted shift preserving operators on $L^{2}(\mathbb{R}^{2n})$}
	
	\author{Radha Ramakrishnan}
	\address{Department of Mathematics, Indian Institute of Technology, Madras, India}
	\email{radharam@iitm.ac.in}
	
	\author{Rabeetha Velsamy}
	\address{Department of Mathematics, Indian Institute of Technology, Madras, India}
	\email{rabeethavelsamy@gmail.com}
	
	
	
	\begin{abstract}
	In the first part of the paper we prove that a twisted shift-invariant subspace of $L^{2}(\mathbb{R}^{2n})$ can be decomposed as a direct sum of mutually orthogonal principal twisted shift-invariant spaces such that the respective system of twisted translates forms a Parseval frame sequence. Later, we introduce twisted preserving operators and the corresponding range operators. We establish that the twisted shift preserving operators and the corresponding range operators simultaneously share some properties in common, namely, self-adjoint, unitary, range of the spectrum and bounded below properties. Finally we prove that the frame operator and its inverse associated with a system of twisted translates of $\{\p_{s}\}_{s\in \z}$ are shift preserving and investigate their corresponding range operators.
	\end{abstract}
	\maketitle
	\section{Introduction}
	Shift-invariant spaces play an important role in many fields such as time-frequency analysis, sampling theory, approximation theory, numerical analysis, electrical engineering and so on. Bownik in \cite{bownik}, characterized shift-invariant system in terms of range functions and obtained characterization for a system of translates on $\mathbb{R}^{n}$ to be a frame sequence and a Riesz sequence.\\
	
	The range function represents the shift-invariant space $V$ as a measurable field of closed subspaces, called fiber spaces, of $l^{2}(\mathbb{Z}^{n})$. The connection between $V$ and its associated range function is obtained through a map $\tau$, which is an isometric isomorphism. In \cite{bownik}, Bownik considered a shift preserving operator acting on a shift-invariant space and studied its properties through the operator called ``range operator''. The shift preserving operators are nothing but those operators which commute with integer translates. On the other hand, a range operator is a representation of a shift preserving operator as a measurable field of linear operators using the $\tau$ map as mentioned above.\\
	
	These works were later generalized to locally compact abelian groups by Cabrelli and Paternostro in \cite{cabrelli} and independently by Kamyabi Gol and Raisi Tousi in \cite{kamyabi,KamyabiTaiwan}. These results were studied for a non-abelian compact group in \cite{Shravan}. Recently a generalized spectral theorem for normal bounded shift preserving operators has been proved by Aguilera et al in \cite{Aguilera}. In another recent work, Barbieri et al \cite{Barbierigppreserving} consider a second countable locally compact abelian group $G$ and operators defined on subspaces of $L^{2}(G)$, which are invariant under the action of a group $\Gamma$, a semi-direct product of a uniform lattice of $G$, with a discrete group of automorphisms. They generalize the notion of shift preserving operators to ``$\Gamma$ preserving operators'' and obtain a spectral decomposition for such operators. Apart from these theortical works the shift preserving operators are useful in signal processing applications, particularly in digital filter modeling.\\
	
	Twisted convolution is a non-standard convolution which arises while transferring the convolution of the Heisenberg group to the complex plane. Under this operation of twisted convolution, $L^{1}(\mathbb{R}^{2n})$ turns out to be a non commutative Banach algebra. Hence the study of (twisted) shift-invariant spaces on $\mathbb{R}^{2n}$ completely differs from the perspective of the usual shift-invariant spaces on $\mathbb{R}^{n}$. There are several works available in the literature in connection with shift-invariant spaces, frames, Riesz bases, wavelets on the Heisenberg group and its variants such as polarized Heisenberg group, A.Weil's abstract Heisenberg group. (See \cite{shannon,Mayeli,arati_orthonormality,arati_muckenhoupt,Arati,saswata_H.G,Das_H.G}). In order to study shift preserving operators in future in the context of Heisenberg group, we first attempt to study shift preserving operators in connection with twisted shift-invariant spaces in this paper. In fact, in order to study problems related to the group Fourier transform on the Heisenberg group $\mathbb{H}^{n}$ whose underlying manifold is $\mathbb{R}^{n}\times\mathbb{R}^{n}\times \mathbb{R}$, an important technique is to take the partial Fourier transform in the central variable and reduce the study to the case $\mathbb{R}^{2n}$. Then the analysis on the Heisenberg group can be investigated by first looking into the analysis based on the Weyl transform and twisted convolution.\\
	
	Recently, in \cite{rabeetha} the authors studied twisted shift-invariant system and obtained characterization theorems for such system to form a frame sequence or a Riesz sequence. In \cite{rabeetha_2} twisted shift-invariant spaces along with frames and Riesz sequences were studied from a different point of view. In fact, Zak transform associated with the Weyl transform was introduced and the characterization for the system of translates to form a frame sequence or a Riesz sequence was obtained using a bracket map arising out of the above Zak transform.\\
	
	In section $3$ of this paper, we introduce the $\j$ map using the Zak transform associated with the Weyl transform on $L^{2}(\mathbb{R}^{2n})$. For $\p\in L^{2}(\mathbb{R}^{2n})$, we mention the characterization for the system of twisted translates $E^{t}(\p)$ to be a Parseval frame for the principal twisted shift-invariant space $V^{t}(\p)$, in terms of the $\j$ map. Subsequently, we obtain a decomposition for a twisted shift-invariant subspace of $L^{2}(\mathbb{R}^{2n})$ as a direct sum of mutually orthogonal principal twisted shift-invariant spaces such that the respective system of twisted translates forms a Parseval frame sequence. In section $4$, we introduce a twisted shift preserving operator and a range operator. Most precisely, we obtain a characterization for a twisted shift preserving operator in terms of a range operator. Further, we establish that the twisted shift preserving operators and the corresponding range operators simultaneously share some properties in common, namely, self-adjoint, unitary, range of the spectrum and bounded below properties. In section $5$, we prove that the frame operator and its inverse associated with a system of twisted translates of $\{\p_{s}\}_{s\in \z}$ are shift preserving. We also prove that the range operators associated with the frame operator and its inverse turn out to be the dual Gramian and its inverse respectively corresponding to the system $\{\j\p_{s}(\cdot, \cdot)\}_{s\in \z}$. 
	\section{Notation and Background}
		Let $\h\neq 0$ be a separable Hilbert space.
	\begin{definition}
		A sequence $\{f_{k} : k\in\z\}$ in $\h$ is called a frame for $\h$ if there exist two constants $A,B>0$ such that
		\begin{align*}
			A\|f\|^{2}\leq \sum_{k\in \z}|\langle f,f_{k}\rangle|^{2} \leq B\|f\|^{2},\hspace{10mm}\forall\ f\in \h.\numberthis \label{frm}
		\end{align*}
	\end{definition}
	If only the right hand side inequality of \eqref{frm} holds, then $\{f_{k} : k\in \z\}$ is called a Bessel sequence. If $A=B$ in \eqref{frm}, then $\{f_{k} : k\in \z\}$ is called a tight frame. If \eqref{frm} holds with $A=B=1$, then $\{f_{k} : k\in \z\}$ is called a Parseval frame. If $\{f_{k} : k\in \z\}$ is a frame for $\overline{\Span\{f_{k} : k\in \z\}}$, then it is called a frame sequence. The frame operator $S:\h\rightarrow \h$ associated with a frame $\{f_{k} : k\in \z\}$ is defined by
	\begin{align*}
		Sf:=\sum\limits_{k\in \z}\langle f,f_{k}\rangle f_{k},\hspace{10mm}\forall\ f\in \h.
	\end{align*}
	It can be shown that $S$ is a bounded, invertible, self-adjoint and positive operator on $\h$. In addition, $\{S^{-1}f_{k} : k\in \z\}$ is also a frame with frame operator $S^{-1}$ and frame bounds $B^{-1}, A^{-1}$.\\
	
	A closed subspace $V\subset L^{2}(\mathbb{R})$ is called a shift-invariant space if $f\in V$ $\implies$ $T_{k}f\in V$ for any $k\in \z$ where $T_{k}$ denotes the translation operator $T_{k}f(y)=f(y-x)$. In particular, for $\p\in L^{2}(\mathbb{R})$, the space $V(\phi)=\overline{\Span\{T_{k}\p : k\in \z\}}$ is called the principal shift-invariant space. We say $\psi\in L^{2}(\mathbb{R})$ is a Parseval frame generator of $V(\p)$ if \[\sum\limits_{k\in\z}|\langle f, T_{k}\psi\rangle |^{2}=\|f\|^{2},\hspace{3mm}\text{for all}~f\in V(\p).\]
	
	We refer to \cite{CN} for the study of frames.\\
	
	Now, we provide the necessary background for the study of the Heisenberg group.\\
	
	The Heisenberg group $\mathbb{H}^{n}$ is a nilpotent Lie group whose underlying manifold is $\mathbb{R}^{n}\times\mathbb{R}^{n}\times\mathbb{R}$ with the group operation defined by $(x,y,t)(u,v,s)=(x+u,y+v,t+s+\frac{1}{2}(u\cdot y-v\cdot x))$ and the Haar measure is the Lebesgue measure $dx\,dy\,dt$ on $\mathbb{R}^{n}\times\mathbb{R}^{n}\times\mathbb{R}$. Using the Schr\"{o}dinger representation $\pi_{\lambda},$ $\lambda\in\mathbb{R}^{\ast},$ given by
	\begin{equation*}
		\pi_{\lambda}(x,y,t)\p(\xi)=e^{2\pi i\lambda t}e^{2\pi i\lambda(x\cdot\xi+\frac{1}{2}x\cdot y)}\p(\xi+y),~\p\in L^{2}(\mathbb{R}^{n}),
	\end{equation*}
	we define the group Fourier transform of $f\in L^{1}(\mathbb{H}^{n})$ as
	\begin{equation*}
		\widehat{f}(\lambda)=\int_{\mathbb{H}^{n}}f(z,t)\pi_{\lambda}(z,t)\,dzdt,~\text{where}~\lambda\in\mathbb{R}^{\ast},
	\end{equation*}
	which is a bounded operator on $L^{2}(\mathbb{R}^{n})$. In otherwords, for $\p\in L^{2}(\mathbb{R}^{n})$, we have
	\begin{equation*}
		\widehat{f}(\lambda)\p=\int_{\mathbb{H}^{n}}f(z,t)\pi_{\lambda}(z,t)\p\,dzdt,
	\end{equation*}
	where the integral is a Bochner integral taking values in $L^{2}(\mathbb{R}^{n}).$ If $f$ is also in $L^{2}(\mathbb{H}^{n})$, then $\widehat{f}(\lambda)$ is a Hilbert-Schmidt operator. Define
	\begin{equation*}
		f^{\lambda}(z)=\int_{\mathbb{R}}e^{2\pi i\lambda t}f(z,t)\,dt,
	\end{equation*}
	which is the inverse Fourier transform of $f$ in the $t$ variable. Then we can write
	\begin{equation*}
		\widehat{f}(\lambda)=\int_{\mathbb{C}^{n}}f^{\lambda}(z)\pi_{\lambda}(z,0)dz.
	\end{equation*}
	Let $g\in L^{1}(\mathbb{C}^{n})$. Define
	\begin{equation*}
		W_{\lambda}(g)=\int_{\mathbb{C}^{n}}g(z)\pi_{\lambda}(z,0)\,dz.
	\end{equation*}
	When $\lambda=1$, it is called the Weyl transform of $g$, denoted by $W(g)$. This can be explicitly written as
	\begin{equation*}
		W(g)\p(\xi)=\int_{\mathbb{R}^{2n}}g(x,y)e^{2\pi i(x\cdot\xi+\frac{1}{2}x\cdot y)}\p(\xi+y)\,dxdy,~\p\in L^{2}(\mathbb{R}^{n}).
	\end{equation*}
	The Weyl transform is an integral operator with kernel $K_{g}(\xi,\eta)=\int_{\mathbb{R}^{n}}g(x,\eta-\xi)e^{\pi ix\cdot(\xi+\eta)}dx$. If $g\in L^{1}\cap L^{2}(\mathbb{C}^{n})$, then $K_{g}\in L^{2}(\mathbb{R}^{2n})$, which implies that $W(g)$ is a Hilbert-Schmidt operator whose norm is given by $\|W(g)\|^{2}_{\mathcal{B}_{2}}=\|K_{g}\|^{2}_{L^{2}(\mathbb{R}^{2n})}$, where $\mathcal{B}_{2}$ is the Hilbert space of Hilbert-Schmidt operators on $L^{2}(\mathbb{R}^{n})$ with inner product $(T,S)=\tr(TS^{\ast})$. The Plancherel formula and the inversion formula for the Weyl transform are given by $\|W(g)\|^{2}_{\mathcal{B}_{2}}=\|g\|^{2}_{L^{2}(\mathbb{C}^{n})}$ and $g(w)=\tr(\pi(w)^{\ast}W(g))$, $w\in \mathbb{C}^{n}$, respectively. For a detailed study of analysis on the Heisenberg group we refer to \cite{follandphase, thangavelu}.
	
	\begin{definition}\cite{saswata}
		Let $\p \in \l$. For $(k,l)\in \mathbb{Z}^{2n},$ the twisted translation $\t\p$ of $\p$, is defined by
		\begin{align*}
			\t\p(x,y)=e^{\pi i(x\cdot l-y\cdot k)}\p(x-k,y-l),\hspace{5mm}(x,y)\in \mathbb{R}^{2n}.
		\end{align*}
	\end{definition}
	Using the definition of twisted translation, we have
	\begin{align*}
		T_{(k_1,l_1)}^tT_{(k_2,l_2)}^t=e^{-\pi i(k_1\cdot l_2-l_1\cdot k_2)}T_{(k_1+k_2,l_1+l_2)}^t,\ \ \ \ \ \forall\ (k_1,l_1),(k_2,l_2)\in\mathbb{Z}^{2n}.\numberthis  \label{comptsttrns}
	\end{align*}
	\begin{lemma}\label{twistker}\cite{saswata}
		Let $\p \in \l$. Then the kernel of the Weyl transform of $\t\p$ satisfies
		\begin{align*}
			K_{\t\p}(\xi,\eta)=e^{\pi i(2\xi+l)\cdot k}K_{\p}(\xi+l,\eta).
		\end{align*}
	\end{lemma}

A closed subspace $V\subset \l$ is said to be twisted shift-invariant space if $f\in V$, then $\t f\in V$, for any $k,l\in \Z$, where $\t f$ denotes the twisted translation of $f$. Let $\A:=\{\p_{s} : s\in \mathbb{Z}\}$ be the countable collection of functions in $\l$ and $E^{t}(\A):=\{\t\p_{s} : k,l\in \Z, s\in \mathbb{Z}\}$. We shall denote $\Span(E^{t}(\A))$ by $U^{t}(\A)$ and $\overline{U^{t}(\A)}$ by $V^{t}(\A)$. When $A=\{\p\}$, it is called the principal twisted shift-invariant space.\\

In \cite{rabeetha_2}, we introduced Zak transform on $\l$ associated with the Weyl transform which is defined as follows : For $\p\in \l$, \[Z_{W}\p(\c,\d\,\e)=\sum\limits_{m\in \Z}K_{\p}(m+\c,\e)e^{-2\pi im\cdot \d},\hspace{3mm}\c,\d\in\T,\e\in \mathbb{R}^{n}.\] Then $Z_{W}$ turns out to be an isometric isomorphism between $\l$ and $L^{2}(\mathbb{T}^{n}\times \T\times \mathbb{R}^{n})$. The operator $Z_{W}$ is called the Weyl-Zak transform and it acts on twisted translation in the following way. \[Z_{W}\t\p(\c,\d,\e)=e^{2\pi i(k\cdot \c+l\cdot \d)}e^{\pi ik\cdot l}Z_{W}\p(\c,\d,\e),\hspace{3mm}k,l\in \Z,\c,\d\in \T,\e\in \mathbb{R}^{n}.\numberthis\label{zaktwist}\] The bracket map $\left[\p,\psi\right]$, for $\p,\psi\in \l$, is defined to be \[\left[\p,\psi\right](\c,\d)=\int_{\mathbb{R}^{n}}Z_{W} \p(\c,\d,\e)\overline{Z_{W} \psi(\c,\d,\e)}\,d\e.\] which satisfies \[\|\left[\p,\psi\right]\|_{L^{1}(\mathbb{T}^{n}\times\mathbb{T}^{n})}\leq \|\p\|_{\l}\|\psi\|_{\l}.\] We refer \cite{rabeetha_2} for further details. We make use of the following results proved in \cite{rabeetha_2} in the sequel. \begin{proposition}\label{zf=rzp}
	Let $\p\in \l$. Then $f\in V^{t}(\p)$ if and only if
	\begin{equation*}
		Z_{W}f(\c,\d,\e)=r(\c,\d)Z_{W}\p(\c,\d,\e),~\text{for a.e.}~\c,\d\in \T, \e\in \mathbb{R}^{n},
	\end{equation*}
	for some $r\in L^{2}(\T \times \T;\left[\p,\p\right]).$ Further, $\|f\|_{L^{2}(\mathbb{R}^{2n})}=\|r\|_{L^{2}(\T \times \T;\left[\p,\p\right])}.$
\end{proposition}
\begin{theorem}\label{frameequiv}
	Let $\p\in \l.$ Then $\{\t\p : k,l\in \Z\}$ is a frame sequence with frame bounds $A,B>0$ iff 
	\begin{equation*}
		0<A\leq \left[\p,\p\right](\c,\d)\leq B<\infty,\hspace{5mm}\text{for a.e.}~(\c,\d)\in\Omega_{\p},
	\end{equation*}
	where $\Omega_{\p}:=\{(\c,\d)\in \mathbb{T}^{n}\times\T : \left[\p,\p\right](\c,\d)\neq 0\}.$
\end{theorem}
\section{The $\j-$ map and twisted shift-invariant spaces}
Now, define $\j : \l\rightarrow \L$ by \[\j f(\c,\d)(\e):=Z_{W}f(\c,\d,\e),\hspace{3mm}f\in \l,\c,\d\in\T,\e\in \mathbb{R}^{n}. \] Then $\j$ is an isometric isomorphism between $\l$ and $L^{2}(\mathbb{T}^{n}\times \T, L^{2}(\mathbb{R}^{n}))$ and $\j$ satisfies \[\j\t f(\c,\d)=e^{2\pi i(k\cdot \c+l\cdot \d)}e^{\pi ik\cdot l}\j f(\c,\d),\hspace{3mm}k,l\in \Z,\c,\d\in \T,\e\in \mathbb{R}^{n}\numberthis\label{jtwist},\] by using \eqref{zaktwist}. \\

Corresponding to this $\j$, we can define a range function as $J(\c,\d)=\overline{\Span\{\j f(\c,\d) : f\in V\}},$ for $\c,\d\in\T$. If $P(\c,\d)$ denotes the orthogonal projection of $\R$ onto $J(\c,\d)$, then the range function $\j$ is said to be measurable provided $(\c,\d)\mapsto \langle P(\c,\d)\Lambda_{1},\Lambda_{2}\rangle$ is measurable for each $\Lambda_{1},\Lambda_{2}\in \R$. The following proposition gives the connection between the image of $V$ under the map $\j$ and the range function $J$.

\begin{proposition}\label{rantau}
	For a measurable range function $J$ define $M_{J}:=\{\Phi \in L^{2}(\mathbb{T}^{n}\times \T, \R) : \Phi(\xi,\d)\in J(\xi,\d)~\text{for a.e.}~\xi,\d\in \mathbb{T}^{n}\}$. Then $M_{\j}$ and $ \j(V)$ satisfy the following properties.
	\begin{enumerate}[(i)]
		\item $M_{J}$ is a closed subspace of $\L$.
		\item $\j(V)$ is a closed subspace of $\L$. Moreover $\j(V)$ is closed under multiplication by exponentials. In other words, $\Phi\in \j(V)$ implies $e^{2\pi i\langle \cdot,k\rangle }\Phi(\cdot)\in \j(V)$, $\forall\ k\in \mathbb{Z}^{n}\times \Z$.
		\item $M_{J}=\j(V)$.
	\end{enumerate}
\end{proposition}
The proof follows similar lines as in the proof of the Proposition $1.5$ in \cite{bownik}.\\ 

Now, using Proposition \ref{zf=rzp}, we obtain the following result. 

\begin{proposition}\label{jf=rjp}
	Let $\p\in \l$. Then $f\in V^{t}(\p)$ if and only if
	\begin{equation*}
		\j f(\c,\d)=r(\c,\d)\j\p(\c,\d),~\text{for a.e.}~\c,\d\in \T, \e\in \mathbb{R}^{n},
	\end{equation*}
	for some $r\in L^{2}(\T \times \T;\left[\p,\p\right]).$
\end{proposition}

\begin{theorem}\label{frameequivinrange}
	The system $E^{t}(\mathcal{A})$ is a frame sequence with frame bounds $A$ and $B$ if and only if the system $\{\j \p_{s}(\xi,\d) : s\in\z\}$ is a frame sequence with the same bounds, for a.e. $\xi,\d\in \T$.
\end{theorem}
The proof follows similar lines as in the proof of Theorem 3.3 in \cite{rabeetha}.

\begin{theorem}\label{parsevalframeequiv}
	Let $\p\in \l$. Then $E^{t}(\p)$ is a Parseval frame for $V^{t}(\p)$ if and if $\|\j \p(\c,\d)\|_{\R}=1$, for a.e. $(\c,\d)\in \Omega_{\p}$, where $\Omega_{\p}=\{(\c,\d)\in \T\times \T : \left[\p,\p\right](\c,\d)\neq 0\}$. 
\end{theorem}
\begin{proof}
	We observe that \[\|\j \p(\c,\d)\|_{\R}^{2}=\int\limits_{\r}|Z_{W}\p(\c,\d,\e)|^{2}d\e=\left[\p,\p\right](\c,\d).\] Then the result follows from Theorem \ref{frameequiv}.
\end{proof}

\begin{remark}
	If $\p$ is a Parseval frame generator of $\v(\p)$ and $m\in L^{2}(\T\times \T)$, then $m\j\p\in \L.$ In fact, using Theorem \ref{parsevalframeequiv}, we get
	\begin{align*}
		\int\limits_{\T}\int\limits_{\T}\|m\j\p(\c,\d)\|^{2}_{\R}\,d\d d\c&=\int\limits_{\T}\int\limits_{\T}|m(\c,\d)|^{2}\|\j\p(\c,\d)\|^{2}_{\R}\,d\d d\c\\
		&=\iint\limits_{\Omega_{\p}}|m(\c,\d)|^{2}\,d\d d\c\leq \|m\|^{2}_{L^{2}(\T\times \T)}<\infty.
	\end{align*}
\end{remark}

\begin{proposition}\label{orthospaceequi}
	The spaces $\v(\p_{1})$ and $\v(\p_{2})$ are orthogonal if and only if \[\langle \j\p_{1}(\c,\d),\j\p_{2}(\c,\d)\rangle_{\R}=0,\] for a.e. $\c,\d\in \T.$
\end{proposition}
\begin{proof}
	By \eqref{comptsttrns}, it is enough to prove that $\v(\p_{1})$ and $\v(\p_{2})$ are orthogonal if and only if \[\langle \p_{1},\t\p_{2}\rangle_{\l}=0,\hspace{2mm}\forall~k,l\in \Z.\] Using the fact that $\j$ is an isometric isomorphism, we get \[\langle \p_{1},\t\p_{2}\rangle_{\l}=\int\limits_{\T}\int\limits_{\T}\langle \j\p_{1}(\c,\d),\j\t\p_{2}(\c,\d)\rangle_{\R}\,d\d d\c.\] Using \eqref{jtwist} and the fact that $\{e^{2\pi i(k\cdot \c+l\cdot \d)}e^{\pi ik\cdot l} : k,l\in\Z\}$ is an orthonormal basis for $L^{2}(\mathbb{T}^{n}\times \mathbb{T}^{n})$, we get \[\langle \p_{1},\t\p_{2}\rangle_{\l}=0,~\forall~k,l\in \Z\hspace{3mm}\text{iff}\hspace{3mm}\langle \j\p_{1}(\c,\d),\j\p_{2}(\c,\d)\rangle_{\R}=0,~\text{for a.e.}~\c,\d\in \T.\]
\end{proof}
\begin{lemma}\label{construofparseval}
	Let $\p\in \l$. Define $\psi\in \l$ by \[\j\psi(\c,\d)=\begin{cases}
	\frac{\j\p(\c,\d)}{\sqrt{\left[\p,\p\right](\c,\d)}} & \text{if}~(\c,\d)\in \Omega_{\p},\\
	0 & \text{otherwise},
	\end{cases}\] where $\Omega_{\p}=\{(\c,\d)\in \T\times \T : \left[\p,\p\right](\c,\d)\neq 0\}$. Then $\{\t\psi : k,l\in \Z\}$ is a Parseval frame for $\v(\p)$. Further, $\v(\p)=\v(\psi)$.
\end{lemma}
\begin{proof}
	Define $\Phi\in\L$ by \[\Phi(\c,\d)=\begin{cases}
	\frac{\j\p(\c,\d)}{\sqrt{\left[\p,\p\right](\c,\d)}} & \text{if}~(\c,\d)\in \Omega_{\p},\\
	0 & \text{otherwise},
	\end{cases}.\] Since $\j$ is surjective, there exists $\psi\in \l$ such that $\j\psi=\Phi$. Moreover, we have \[\left[\psi,\psi\right](\c,\d)=\|\j \psi(\c,\d)\|^{2}_{\l}=\chi_{\Omega_{\p}}(\c,\d),\] which leads to $\Omega_{\psi}=\Omega_{\p}$. Then it follows from Theorem \ref{parsevalframeequiv}, that the system $\{\t\psi : k,l\in \Z\}$ is a Parseval frame for $\v(\psi)$. In order to complete the proof, we need to show that $\v(\psi)=\v(\p)$. Define $r : \T\times \T\rightarrow \mathbb{C}$ by $r(\c,\d)=(\left[\p,\p\right](\c,\d))^{-\frac{1}{2}}\chi_{\Omega_{\p}}(\c,\d).$ Clearly $r\in \Ll$ and $\j\psi(\c,\d)=r(\c,\d)\j\p(\c,\d),$ for $\psi\in \l$. Then, by using Proposition \ref{jf=rjp}, $\psi\in \v(\p)$, which implies that $\v(\psi)\subset \v(\p)$. Now, define $r^{'} : \T\times \T\rightarrow \mathbb{C}$ by $r^{'}(\c,\d)=(\left[\p,\p\right](\c,\d))^{\frac{1}{2}}.$ Then $r^{'}\in L^{2}(\T\times \T;\left[\psi,\psi\right])$ and $\j\phi(\c,\d)=r(\c,\d)\j\psi(\c,\d),$ for $\p\in \l$, which leads to $\v(\phi)\subset \v(\psi)$, thus proving our assertion.
	\end{proof}

\begin{theorem}\label{decomposition}
	If $V$ is a twisted shift-invariant subspace in $\l$, then there exists a collection of functions $\{\p_{i}\}_{i\in \N}$ in $\l$ such that $V=\bigoplus\limits_{i\in\N}\v(\p_{i})$, where each $\p_{i}$ is a Parseval frame generator of $\v(\p_{i})$, for $i\in\N$. Moreover, if $f\in V$, then $\|f\|^{2}_{2}=\sum\limits_{i\in \mathbb{N}}\|r_{i}\|^{2}_{L^{2}(\mathbb{T}^{n}\times \mathbb{T}^{n}; \left[\phi_{i},\phi_{i}\right])},$ where $r_{i}\in L^{2}(\mathbb{T}^{n}\times \mathbb{T}^{n}; \left[\phi_{i},\phi_{i}\right]).$
\end{theorem}
\begin{proof}
	Let $0\neq \psi\in V$. Then, by applying Lemma \ref{construofparseval}, we obtain $\p\in \v(\psi)$ such that $\p$ be a parseval frame generator of $\v(\p)$. Since $\v(\p)$ is a closed subspace of $V$, we have $V=\v(\p)\oplus\v(\p)^{\perp}$. By applying the above argument to $\v(\p)^{\perp}$ in place of $V$, we obtain a collection of mutually orthogonal principal twisted shift-invariant subspaces of $V$. By Zorn's lemma, there exists a maximal collection $\{\v(\p_{\alpha})\}_{\alpha\in I}$ such that $\p_{\alpha}$ is a Parseval frame generator of $\v(\p_{\alpha})$, for each $\alpha\in I.$ Suppose $V\neq \bigoplus\limits_{\alpha\in I}\v(\p_{\alpha})$, then there exists $0\neq g\in V$ such that the principal twisted shift-invariant space $\v(g)$ is orthogonal to $\v(\p_{\alpha})$, for every $\alpha\in I$, which is a contradiction to the maximality of the collection. Hence, we obtain $V=\bigoplus\limits_{\alpha\in I}\v(\p_{\alpha})$. Clearly $I$ is countable as $\l$ is separable. Let $f\in V$. Then $f=\sum\limits_{i\in \mathbb{N}}f_{i}$, where $f_{i}\in V^{t}(\phi_{i})$. Furthermore, $\|f\|^{2}_{L^{2}(\mathbb{R}^{2n})}=\sum\limits_{i\in \mathbb{N}}\|f_{i}\|^{2}_{L^{2}(\mathbb{R}^{2n})}$. From the Proposition \ref{zf=rzp}, we get $\|f_{i}\|^{2}_{L^{2}(\mathbb{R}^{2n})}=\|r_{i}\|^{2}_{L^{2}(\mathbb{T}^{n}\times \mathbb{T}^{n},\left[\phi_{i},\phi_{i}\right])},$ where $r_{i}\in L^{2}(\mathbb{T}^{n}\times \mathbb{T}^{n},\left[\phi_{i},\phi_{i}\right])$, which gives that $\|f\|^{2}_{L^{2}(\mathbb{R}^{2n})}=\sum\limits_{i\in \mathbb{N}}\|r_{i}\|^{2}_{L^{2}(\mathbb{T}^{n}\times \mathbb{T}^{n},\left[\phi_{i},\phi_{i}\right])}.$  
\end{proof}

\section{Characterization for a twisted shift preserving operator}
In this section, we introduce a twisted shift preserving operator and find a necessary and sufficient condition for a bounded linear operator to be a twisted shift preserving operator.
\begin{definition}
	A bounded linear operator $U : \l\rightarrow\l$ is called a twisted shift preserving operator if $U\t=\t U$ for all $k,l\in \Z,$ where $\t$ is a twisted translation operator.
\end{definition}
\begin{example}\label{example}
	Let $\p\in L^{\infty}(\mathbb{R}^{2})$ be a ``$1$'' periodic function. Define a multiplication operator $M_{\p} : L^{2}(\mathbb{R}^{2})\rightarrow L^{2}(\mathbb{R}^{2})$ by $M_{\phi}f=\p f$. Now,
	\begin{align*}
		\t M_{\p}f(x,y)&=e^{\pi i(xl-yk)}M_{\p}f(x-k,y-l)\\
		&=e^{\pi i(xl-yk)}\p(x-k,y-l)f(x-k,y-l)\\
		&=e^{\pi i(xl-yk)}\p(x,y)f(x-k,y-l)\\
		&=\p(x,y)\t f(x,y)\\
		&=M_{\p}\t f(x,y),
	\end{align*}
for any $f\in L^{2}(\mathbb{R}^{2})$, which shows that $M_{\p}$ is a twisted shift preserving operator.
\end{example}
\begin{definition}
	Suppose $V$ is a twisted shift-invariant subspace of $\l$ with the range function $J$ and associated projection $P$. A range operator on $J$ is a mapping \[R : \T\times \T\rightarrow \{\text{bounded linear operators on closed subspaces of}~\R\},\] so that the domain of $R(\c,\d)$ equals $J(\c,\d)$ for a.e. $\c,\d\in \T.$
\end{definition} The operator $R$ is called measurable if $(\c,\d)\mapsto \langle R(\c,\d)P(\c,\d)a, b\rangle $ is a measurable scalar function for all $a,b\in \R.$
Now, we prove the following lemma which can be used to prove the characterization theorem for a twisted shift preserving operator.

\begin{lemma}\label{characlemma}
Let $\p\in \l$ be a Parseval frame generator of $\v(\p)$ and $U : \v(\p)\rightarrow\l$ be a twisted shift preserving operator. Then for every $m\in L^{2}(\T\times \T),$ \[(\j\circ U\circ \J)(m\j \p)(\c,\d)=m(\c,\d)(\j\circ U)\p(\c,\d),\] for a.e. $\c,\d\in \T.$
\end{lemma}
\begin{proof}
	As $U$ is a twisted shift preserving operator, by making use of \eqref{jtwist}, we get \begin{align*}
		(\j \circ U\circ \J)(e^{2\pi i(k\cdot +l\cdot)}\j \p)(\c,\d)&=(\j \circ U\circ \J)(e^{-\pi ik\cdot l}\j \t\p)(\c,\d)\\
		&=e^{-\pi ik\cdot l}(\j\circ U\circ \t)\p(\c,\d)\\
		&=e^{-\pi ik\cdot l}\j\t(U\p)(\c,\d)\\
		&=e^{2\pi i(k\cdot\c +l\cdot\d)}(\j\circ U)\p(\c,\d).\numberthis\label{characlemmaexpo}
	\end{align*}
	By linearity, \eqref{characlemmaexpo} holds for all trigonometric polynomials in $L^{2}(\T\times \T)$. That is, \[(\j \circ U\circ \J)(p\j \p)(\c,\d)=p(\c,\d)(\j\circ U)\p(\c,\d).\] Moreover, since $\j$ is an isometric isomorphism, we get \begin{align*}
		\I\I |p(\c,\d)|^{2}\|(\j\circ U)\p(\c,\d)\|^{2}_{\R}\,d\d d\c
		&\leq \|U\|^{2}\|P\j \p\|^{2}_{\L}\\
		&=\|U\|^{2}\I\I |p(\c,\d)|^{2}\|\j \p(\c,\d)\|^{2}_{\R}\,d\d d\c,
	\end{align*}
	for any trigonometric polynomial $p$ in $L^{2}(\T\times \T).$ Let $r\in L^{\infty}(\T\times \T)\subset L^{2}(\T\times \T)$. By Lusin's theorem, there exists a sequence of polynomials $\{p_{i}\}_{i\in\N}$ such that $\|p_{i}\|_{L^{\infty}(\T\times \T)}\leq \|r\|_{L^{\infty}(\T\times \T)}$ and $p_{i}(\c,\d)\to r(\c,\d)$ as $i\to \infty,$ for a.e. $\c,\d \in \T.$ Then, by Lebsegue dominated convergence theorem, we get \[\I\I |r(\c,\d)|^{2}\|(\j\circ U)\p(\c,\d)\|^{2}_{\R}\,d\d d\c\leq \|U\|^{2}\I\I |r(\c,\d)|^{2}\|\j \p(\c,\d)\|^{2}_{\R}\,d\d d\c.\] Since the above inequality is true for any $r\in L^{\infty}(\T\times \T)$, we get \[\|(\j\circ U)\p(\c,\d)\|_{\R}\leq \|U\|\|\j\p(\c,\d)\|_{\R},\numberthis\label{jou<j}\] for a.e. $\c,\d\in \T.$ Let $m\in L^{2}(\T\times \T)$. Then there exists a sequence of polynomials $\{p_{i}\}_{i\in \N}$ such that $p_{i}(\c,\d)\to m(\c,\d)$ as $i\to \infty,$ for a.e. $\c,\d \in \T.$ Further, $(\j\circ U\circ \J)(p_{i}\j\p)(\c,\d) \to (\j\circ U\circ \J)(m\j\p)(\c,\d)$ as $i\to \infty$ for a.e. $\c,\d\in \T$. Thus it follows from \eqref{jou<j} that $p_{i}(\c,\d)(\j\circ U)\p(\c,\d)\to m(\c,\d)(\j\circ U)\p(\c,\d)$ as $i\to \infty$ for a.e. $\c,\d\in \T$. This in turn implies that \[(\j\circ U\circ \J)(m\j \p)(\c,\d)=m(\c,\d)(\j\circ U)\p(\c,\d),\] for a.e. $\c,\d\in \T$, proving our assertion.  
\end{proof}
\begin{theorem}\label{characterizationtheorem}
	Suppose $V\subset \l$ is a twisted shift-invariant subspace and J is its associated range function. For every twisted shift preserving operator $U : V\rightarrow \l$, there exists a measurable range operator $R$ on $J$ such that \[(\j\circ U)f(\c,\d)=R(\c,\d)(\j f(\c,\d)),\hspace{3mm}\text{for a.e.}~\c,\d\in \T,f\in V.\numberthis\label{shiftoperaequivcondi}\] Conversely, given a measurable range operator $R$ on $J$ with $\esssup\limits_{\c,\d\in \T}\|R(\c,\d)\|_{\R}<\infty,$ there is a twisted shift preserving operator $U : V\rightarrow \l$ such that \eqref{shiftoperaequivcondi} holds. The correspondance between $U$ and $R$ is one-one under the usual conventiom that the range operators are identified if they are equal a.e.
\end{theorem}
\begin{proof}
	As a consequence of Theorem \ref{decomposition} and Theorem \ref{parsevalframeequiv}, there exists a sequence of Parseval frame generators $\{\p_{i}\}_{i\in \N}$ in $\l$ such that $V=\bigoplus\limits_{i\in\N}\v(\p_{i})$ and for each $i\in \N$, $\|\j\p_{i}(\c,\d)\|_{\R}=1,$ for a.e. $(\c,\d)\in \Omega_{\p_{i}}$. For $m\in \N$, define a twisted shift-invariant space $V_{m}$ by $V_{m}=\bigoplus\limits_{i=1}^{m}\v(\p_{i})$. Let $J_{m}$ be its associated range function. Futher, the collection $\{\t\p_{i} : k,l\in\Z, 1\leq i\leq m\}$ is a Parseval frame for $V_{m}$ and hence $\{\j\p_{i}(\c,\d) : 1\leq i\leq m\}$ is a Parseval frame for $J_{m}(\c,\d)$, using Theorem \ref{frameequivinrange}. Define $R_{m}(\c,\d) : J_{m}(\c,\d)\rightarrow \R$ by \[R_{m}(\c,\d)\bigg(\sum\limits_{i=1}^{m}\alpha_{i}\j\p_{i}(\c,\d)\bigg)=\sum\limits_{i=1}^{m}\alpha_{i}(\j\circ U)\p_{i}(\c,\d).\numberthis\label{Rm}\] But $\{\j\p_{i}(\c,\d) : 1\leq i\leq m\}\setminus\{0\}$ is an orthonormal basis for $J_{m}(\c,\d)$ for a.e. $\c,\d\in \T$. (See Proposition \ref{orthospaceequi}). By \eqref{jou<j}, $R_{m}(\c,\d)$ is well defined. Choose $f\in V_{m}$. Then $f=\sum\limits_{i=1}^{m}f_{i}$ for some $f_{i}\in \v(\p_{i}),$ where $1\leq i\leq m.$ Then by Proposition \ref{jf=rjp}, we have $\j f=\sum\limits_{i=1}^{m}\j f_{i}=\sum\limits_{i=1}^{m}r_{i}\j\p_{i},\hspace{3mm}\text{for some}~r_{i}\in L^{2}(\T\times \T).$ Hence it follows from Lemma \ref{characlemma} and \eqref{Rm} that 
	\begin{align*}
		(\j \circ U)f(\c,\d)&=(\j\circ U\circ \J)(\j f)(\c,\d)\\
		&=\sum\limits_{i=1}^{m}(\j\circ U\circ \J)(r_{i}\j\p_{i})(\c,\d)\\
		&=\sum\limits_{i=1}^{m}r_{i}(\c,\d)(\j\circ U)\p_{i}(\c,\d)\\
		&=\sum\limits_{i=1}^{m}r_{i}(\c,\d)R_{m}(\c,\d)(\j\p_{i}(\c,\d))\\
		&=R_{m}(\c,\d)(\j f(\c,\d)),
	\end{align*} 
for any $f\in V_{m}$. Then one can easily show that $R_{m}$ is measurable. Now, we shall show that \[\|R_{m}(\c,\d)\|\leq \|U\|,\hspace{3mm}\text{for a.e.}~\c,\d\in \T.\numberthis\label{Rmbounded}\]
Suppose \eqref{Rmbounded} does not hold, then there exists a positive Lebesgue measurable set $D\subset \T\times \T$ such that $R_{m}(\c,\d)$ is not bounded by $\|U\|$ for all $(\c,\d)\in D$. In otherwords, there exists $\psi(\c,\d)\in J_{m}(\c,\d)\subset \R$ such that $\|R_{m}(\c,\d)(\psi(\c,\d))\|_{\R}> \|U\|,$ with $\|\psi(\c,\d)\|_{\R}=1$, for $(\c,\d)\in D.$ For $\c,\d\in \T,$ define $\Psi(\c,\d)\in J_{m}(\c,\d)$ by $\Psi(\c,\d)=\psi(\c,\d)\chi_{D}(\c,\d).$ From Proposition \ref{rantau}, there exists $f\in V_{m}$ such that $\j f=\Psi$. But \begin{align*}
	\|Uf\|^{2}_{\l}&=\|(\j\circ U)f\|^{2}_{\L}\\
	&=\I\I\|(\j\circ U)f(\c,\d)\|^{2}_{\R}\,d\d d\c\\
	&=\I\I\|R_{m}(\c,\d)(\j f(\c,\d))\|^{2}_{\R}\, d\d d\c\\
	&=\I\I\|R_{m}(\c,\d)(\psi(\c,\d)\chi_{D}(\c,\d))\|^{2}_{\R}\, d\d d\c\\
	&=\iint\limits_{D}\|R_{m}(\c,\d)(\psi(\c,\d))\|^{2}_{\R}\, d\d d\c\\
	&>\|U\|^{2}\iint\limits_{D}\,d\d d\c\\
	&=\|U\|^{2}\int\limits_{\T}\I\|\psi(\c,\d)\chi_{D}(\c,\d)\|^{2}_{\R}\,d\d d\c\\
	&=\|U\|^{2}\|\j f\|^{2}_{\L}=\|U\|^{2}\|f\|^{2}_{\l}.\\
\end{align*}
This implies that $\|Uf\|>\|U\|\|f\|$, which is contradiction to $\|Uf\|_{\l}\leq \|U\|\|f\|_{\l}$, proving \eqref{Rmbounded}. For any $i\leq m,~m\in \N,$ we have $R_{i}(\c,\d)=R_{m}(\c,\d)|_{J_{i}(\c,\d)}$. We define $R(\c,\d) : \bigcup\limits_{m\in\N}J_{m}(\c,\d)\rightarrow \R$ by $R(\c,\d)(a)=R_{m}(\c,\d)(a)\hspace{3mm}\text{for}~a\in J_{m}(\c,\d),~m\in \N.$ Then for any $m\in\N,$ we have $\|R_{m}(\c,\d)\|\leq \|U\|$ for a.e. $\c,\d\in\T,$ which implies that $\|R(\c,\d)(a)\|_{\R}\leq \|U\|\|a\|_{\R}$ for $a\in \bigcup\limits_{m\in \N}J_{m}(\c,\d)$. Since $\overline{\bigcup\limits_{m\in\N}J_{m}(\c,\d)}=J(\c,\d)$, we can extend $R(\c,\d)$ uniquely to $R(\c,\d) : J(\c,\d)\rightarrow \R$ with $\|R(\c,\d)\|\leq \|U\|$. Let $f\in V$. Then, we can find a sequence $f_{m}\in V_{m}$ such that $f_{m}\rightarrow f$ in $\l$. Moreover, $\j f_{m}(\c,\d) \to \j f(\c,\d)$ and $(\j\circ U) f_{m}(\c,\d) \to (\j\circ U) f(\c,\d)~\text{as}~m\to \infty,\hspace{3mm}\text{for a.e.}~\c,\d\in \T.$ Thus $R(\c,\d)(\j f_{m}(\c,\d)) \to R(\c,\d)(\j f(\c,\d))~\text{as}~m\to \infty.$ Since $(\j\circ U)f_{m}(\c,\d)=R(\c,\d)(\j f_{m}(\c,\d))$, for all $m\in \N$, we obtain $(\j\circ U)f(\c,\d)=R(\c,\d)(\j f(\c,\d)),\hspace{3mm}\forall~f\in V.$ Conversely, let $R$ be a measurable range operator on $J(\c,\d)$ with $\esssup\limits_{\c,\d\in \T}\|R(\c,\d)\|<\infty.$ Define $U : V\rightarrow \l$ by \[U(f)=\J F,\hspace{3mm}\text{where}~F(\c,\d)=R(\c,\d)(\j f(\c,\d)).\] The linearity of $U$ follows from the linearity of $R(\c,\d).$ Now, we shall show that $U$ is bounded and twisted shift preserving operator. Consider,
\begin{align*}
	\|U(f)\|^{2}_{\l}&=\|(\j \circ U)f\|^{2}_{\L}\\
	&=\|F\|^{2}_{\L}\\
	&=\I\I \|R(\c,\d)(\j f(\c,\d))\|^{2}_{\R}\,d\d d\c\\
	&\leq \bigg(\esssup\limits_{\c,\d\in \T}\|R(\c,\d)\|\bigg)^{2}\|\j f\|^{2}_{\L}\\
	&=\bigg(\esssup\limits_{\c,\d\in \T}\|R(\c,\d)\|\bigg)^{2}\|f\|^{2}_{\l}.
\end{align*} Hence $U$ is a bounded operator with $\|U\|\leq\esssup\limits_{\c,\d\in \T}\|R(\c,\d)\|$. Further, for any $k,l\in \Z,f\in V,$ we get $(\j \circ U)\t f(\c,\d)=R(\c,\d)(\j \t f(\c,\d))=\j(\t(Uf))(\c,\d),$
for a.e. $\c,\d\in \T,$ which means that $\j (U\t f)=\j (\t Uf)$. Since $\j$ is injective, $U\t=\t U$. The one-to-one correspondance follows by \eqref{shiftoperaequivcondi}. 
\end{proof}

\begin{example}
	Define $\p(x,y)=e^{2\pi i y}$, $x,y\in \mathbb{R}$. Then $\p\in L^{\infty}(\mathbb{R}^{2})$ and it is ``$1$'' periodic function. Then by Example \ref{example}, $M_{\p}$ is a twisted shift preserving operator. By Theorem \ref{characterizationtheorem}, there exists a range operator $R$ such that $(\j\circ M_{\p})f(\c,\d)=R(\c,\d)(\j f(\c,\d))$, for a.e $\c,\d\in \mathbb{T}$, for any $f\in L^{2}(\mathbb{R}^{2})$. Our aim is to find $R$ explicitly. Towards this end, we compute the kernel of the Weyl transform of $M_{\p}f$. For $\c,\e\in \mathbb{R}$, \begin{align*}
		K_{M_{\p}f}(\c,\e)&=\int\limits_{\mathbb{R}}M_{\p}f(x,\eta-\c)e^{\pi ix(\eta+\c)}\, dx\\
		&=\int\limits_{\mathbb{R}}\p(x,\eta-\c)f(x,\eta-\c)e^{\pi ix(\eta+\c)}\, dx\\
		&=\int\limits_{\mathbb{R}}e^{2\pi i(\eta-\c)}f(x,\eta-\c)e^{\pi ix(\eta+\c)}\, dx\\
		&=e^{2\pi i(\eta-\c)}\int\limits_{\mathbb{R}}f(x,\eta-\c)e^{\pi ix(\eta+\c)}\, dx\\
		&=e^{2\pi i(\eta-\c)}K_{f}(\c,\e).
	\end{align*}
Thus \begin{align*}
	\j(M_{\p}f)(\c,\d)(\eta)&=Z_{W}M_{\p}f(\c,\d,\e)\\
	&=\sum\limits_{m\in \z}K_{M_{\p}f}(\c+m,\e)e^{-2\pi im\d}\\
	&=e^{2\pi i(\eta-\c)}\sum\limits_{m\in \z}K_{f}(\c+m,\e)e^{-2\pi im\d}\\
	&=e^{2\pi i(\eta-\c)}Z_{W}f(\c,\d,\e)\\
	&=e^{2\pi i(\eta-\c)}\j f(\c,\d)(\e).
\end{align*}
Therefore the range operator $R$ for $M_{\p}$ is given by $R(\c,\d)(\cdot)(\e)=e^{2\pi i(\eta-\c)}(\cdot)(\e)$.
\end{example}

\begin{theorem}
	Suppose $V\subset \l$ is a twisted shift-invariant subspace, $J$ is its associated range function and $U$ is a twisted shift preserving operator with its corresponding range operator $R$. Then $U$ is bounded below with a constant $c>0$ if and only if $R(\c,\d)$ is bounded below with a constant $c>0$, for a.e. $\c,\d\in \T.$
\end{theorem}
\begin{proof}
	Assume that $R(\c,\d)$ is bounded below with a constant $c>0$, for a.e. $\c,\d\in \T$. Let $f\in V.$ Consider, 
	\begin{align*}
		\|U(f)\|^{2}_{\l}&=\|(\j\circ U)f\|^{2}_{\L}\\
		&=\I\I\|R(\c,\d)(\j f(\c,\d))\|^{2}_{\R}\,d\d d\c\\
		&\geq c^{2}\|\j f\|^{2}_{\L}=c^{2}\|f\|^{2}_{\l}.
	\end{align*}
	 Since $f\in V$ is arbitrary, $U$ is bounded below by $c>0$. Conversely assume that $U$ is bounded below. The proof of $R(\c,\d)$ is bounded below makes use of similar arguments as that of $R_{m}(\c,\d)$ is bounded above in the proof of Theorem \ref{characterizationtheorem}. 
\end{proof}
\begin{theorem}\label{Uproperties}
	Suppose $V\subset \l$ is a twisted shift-invariant subspace, $J$ is its associated range function and $U : V\rightarrow V$ is a twisted shift preserving operator with its corresponding range operator $R$. Then the following holds. \begin{enumerate}
		\item [(i)] The adjoint operator $U^{\ast} : V \rightarrow V$ is twisted shift preserving operator and its corresponding range operator $R^{\ast}$ is given by $R^{\ast}(\c,\d)=R(\c,\d)^{\ast},$ for a.e. $\c,\d\in \T.$
		\item [(ii)] Let $A\leq B$ be two real numbers. Then $U$ is self-adjoint and $\sigma(U)\subset [A,B]$ if and only if $R(\c,\d)$ is self-adjoint and $\sigma(R(\c,\d))\subset [A,B]$, for a.e. $\c,\d\in \T.$
		\item [(iii)] U is unitary if and only if $R(\c,\d)$ is uniatry for a.e. $\c,\d\in \T.$
	\end{enumerate}
\end{theorem}
\begin{proof}
	(i) We observe that $\langle U^{\ast}\t f, g\rangle=\langle f, T^{t}_{(-k,-l)}Ug\rangle=\langle f, UT^{t}_{(-k,-l)}g\rangle=\langle \t U^{\ast}f, g\rangle.$ for all $f,g\in \l$ and $k,l\in \Z$ from which it follows that $U^{\ast}\t=\t U^{\ast}$. Define an operator $R^{\ast}$ by $R^{\ast}(\c,\d)=R(\c,\d)^{\ast}$. Since $R$ is measurable, $R^{\ast}$ is measurable. Moreover, $\esssup\limits_{\c,\d\in \T}\|R^{\ast}(\c,\d)\|=\esssup\limits_{\c,\d\in \T}\|R(\c,\d)\|<\infty.$ By Theorem \ref{characterizationtheorem}, there exists a twisted shift preserving operator $W$ such that $(\j\circ W)f(\c,\d)=R^{\ast}(\c,\d)(\j f(\c,\d))$ for a.e. $\c,\d\in \T$, for all $f\in V.$ Let $f,g\in \l$. Then $U^{\ast}=W$. In fact, \begin{align*}
		\langle U^{\ast} f,g \rangle_{\l}&=\langle f, Ug\rangle_{\l}\\
		&=\I\I \langle \j f(\c,\d), R(\c,\d)(\j g(\c,\d)) \rangle_{\R}\, d\d d\c\\
		&=\I\I \langle R(\c,\d)^{\ast}(\j f(\c,\d)), \j g(\c,\d) \rangle_{\R}\, d\d d\c\\
		&=\I\I \langle (\j\circ W)f(\c,\d), \j g(\c,\d) \rangle_{\R}\, d\d d\c\\
		&=\langle Wf, g\rangle,
	\end{align*}
	thus proving our assertion.\\
	
	(ii) By making use of \eqref{shiftoperaequivcondi}, we can show that $U$ is self-adjoint if and only if $R(\c,\d)$ is self-adjoint for a.e. $\c,\d\in \T$. Assume that $U$ is self-adjoint and $\sigma(U)\subset [A,B]$. In otherwords, $U-A$ and $B-U$ are positive operators. We claim that $\sigma(R(\c,\d))\subset [A,B]$ for a.e. $\c,\d\in \T$. It is sufficient to show that $R(\c,\d)-A$ and $B-R(\c,\d)$ are positive operators, for a.e. $\c,\d\in \T.$ Suppose it were not true, then there would exist a positive Lebesgue measurable set $D\subset \T\times \T$ such that either $R(\c,\d)-A\hspace{3mm}\text{is not a positive operator}$ or $B-R(\c,\d)\hspace{3mm}\text{is not a positive operator}$ for $(\c,\d)\in D$. Suppose $R(\c,\d)-A\hspace{3mm}\text{is not a positive operator}$, then there exists $\psi(\c,\d)\in J(\c,\d)$ such that $\langle R(\c,\d)(\psi(\c,\d)), \psi(\c,\d)\rangle < A\langle \psi(\c,\d),\psi(\c,\d)\rangle$ for $\c,\d\in D.$ Without loss of generality assume that $\|\psi(\c,\d)\|_{\R}=1$ on $(\c,\d)\in D$. Now, define $\Psi(\c,\d)\in J(\c,\d)$ such that $\Psi(\c,\d)=\psi(\c,\d)\chi_{D}(\c,\d)$. By using Proposition \ref{rantau}, there exists $f\in V$ such that $\j f(\c,\d)=\Psi(\c,\d)$. Consider,
	\begin{align*}
		\langle Uf, f\rangle &=\I\I \langle (\j\circ U)f(\c,\d), \j f(\c,\d)\rangle \, d\d d\c\\
		&=\I\I \langle R(\c,\d)(\j f(\c,\d)), \j f(\c,\d)\rangle \, d\d d\c\\
		&=\iint\limits_{D} \langle R(\c,\d)(\psi(\c,\d)), \psi(\c,\d)\rangle \, d\d d\c\\
		&<A\iint\limits_{D} \langle \psi(\c,\d),\psi(\c,\d)\rangle \, d\d d\c\\
		&=A\I\I \langle \j f(\c,\d), \j f(\c,\d)\rangle \,d\d d\c,
	\end{align*} which is a contradiction to $U-A$ is a positive operator. Conversely assume that $R(\c,\d)$ is self adjoint and $\sigma(R(\c,\d))\subset [A,B]$ for a.e. $\c,\d\in \T.$ That is, \[A\|\j f(\c,\d)\|^{2}_{\R}\leq \langle R(\c,\d)(\j f(\c,\d)), \j f(\c,\d)\rangle_{\R}\leq B\|\j f(\c,\d)\|^{2}_{\R},\] for a.e. $\c,\d\in \T$, for any $f\in V$. Integrating the above inequality with respect to $\c,\d\in \T$, we get \[A\|\j f\|^{2}_{\L}\leq \I\I\langle (\j\circ U)f(\c,\d), \j f(\c,\d)\rangle_{\R}\,d\d d\c\leq B\|\j f\|^{2}_{\L}.\] As $\j$ is isometry, we get $A\|f\|^{2}\leq \langle Uf,f\rangle \leq B\|f\|^{2},$ for any $f\in V$ which shows that $U$ is self adjoint and $\sigma(U)\subset [A,B].$\\
	
	(iii) Since $U$ and $U^{\ast}$ are twisted shift preserving operators, $UU^{\ast}$ and $U^{\ast}U$ are twisted shift preserving operators. By making use of Theorem \ref{characterizationtheorem}, we can show that the range operators for $UU^{\ast}$ and $U^{\ast}U$ are given by $R(\c,\d)R^{\ast}(\c,\d)$ and $R^{\ast}(\c,\d)R(\c,\d)$ respectively, for a.e. $\c,\d\in \T$. But $\sigma(UU^{\ast})=\sigma(U^{\ast}U)=\sigma(I)=\{ 1 \}$. Then it follows from (ii) that $\sigma(R(\c,\d)R^{\ast}(\c,\d))=\sigma(R^{\ast}(\c,\d)R(\c,\d))=\sigma(I_{(\c,\d)})=\{ 1 \}$, where $I_{(\c,\d)}$ is the identity operator on $J(\c,\d)$, proving that $R(\c,\d)$ is unitary for a.e. $\c,\d\in \T$. The converse follows in a similar direction.
\end{proof}

\section{Frame operator as a twisted shift preserving operator}
For $\c,\d \in\T$, define $H(\c,\d) : \ell^{2}(\mathbb{Z})\rightarrow \R$ by
$$H(\xi,\d)(\{c_{s}\})=\sum\limits_{s\in\mathbb{Z}}c_{s}\j \p_{s}(\xi,\d).$$
If $H(\xi,\d)$ defines a bounded operator from $\ell^{2}(\mathbb{Z})$ into $\R$ then its adjoint $H(\xi,\d)^{\ast}:\R\rightarrow \ell^{2}(\mathbb{Z})$ is given by
\begin{align*}
	H(\xi,\d)^{\ast}\p=\Big\{\big<\p,\j\p_{s}(\xi,\d)\big>_{\R}\Big\}_{s\in\mathbb{Z}}.
\end{align*}
The Gramian associated with the system $\{\j \p_{s}(\xi,\d) : s\in\mathbb{Z}\}$, denoted by $G(\xi,\d)$, is defined by $G(\xi,\d)=H(\xi,\d)^{\ast}H(\xi,\d)$. The corresponding dual Gramian, denoted by $\widetilde{G}(\xi,\d)$, is defined by $\widetilde{G}(\xi,\d)=H(\xi,\d)H(\xi,\d)^{\ast}$.\\

In the following theorems, we prove that the frame operator and its corresponding inverse are twisted shift preserving operators. We also find their corresponding range operators.
\begin{lemma}\label{Hlemma}
	Let $f\in L^{2}(\mathbb{R}^{2n})$ and the system $\{\t\p_{s} : k,l\in \Z,s\in \z\}$ be a frame sequence. Then the following equality holds.
	\begin{equation*}
	\sum\limits_{(k,l,s)\in \mathbb{Z}^{2n+1}}\big|\big<f,T^{t}_{(k,l)}\p_{s}\big>\big|^{2}=\I\I \sum\limits_{s\in\z}|\langle \j f(\c,\d), \j \p_{s}(\c,\d)\rangle|^{2}\,d\d d\c.
	\end{equation*}
	The proof follows along the same lines as in Lemma 3.2 in \cite{rabeetha}.
\end{lemma}
\begin{theorem}
	Let $E^{t}(\A)$ be a Bessel sequence and $S$ be the frame operator for the corresponding system $E^{t}(\A)$. Then $S$ is a twisted shift preserving operator. The range operator associated with $S$ is given by $R(\c,\d)=\widetilde{G}(\c,\d)$, where $\widetilde{G}(\c,\d)$ is the dual Gramian operator associated with the system $\{\j \p_{s}(\c,\d) : s\in \z\}$, for a.e. $\c,\d\in \T.$
\end{theorem}
\begin{proof}
	The frame operator $S$ for the system $E^{t}(\A)$ is given by \[Sf=\sum\limits_{(k,l,s)\in\mathbb{Z}^{2n+1}}\langle f, \t\p_{s}\rangle_{\l}\t\p_{s}\hspace{3mm}\forall f\in \v(\A).\] Consider, for $f\in \v(\A)$ and $k^{'},l^{'}\in \Z$. \begin{align*}
		T^{t}_{(k^{'},l^{'})} Sf
		&=\sum\limits_{(k,l,s)\in\mathbb{Z}^{2n+1}}\langle f, \t\p_{s}\rangle e^{-\pi i(k^{'}\cdot l-k\cdot l^{'})}	T^{t}_{(k^{'}+k,l^{'}+l)}\p_{s}\\
			&=\sum\limits_{(k,l,s)\in\mathbb{Z}^{2n+1}}\langle f, T^{t}_{(k-k^{'},l-l^{'})}\p_{s}\rangle e^{-\pi i(k^{'}\cdot (l-l^{'})-(k-k^{'})\cdot l^{'})}	T^{t}_{(k,l)}\p_{s}\\
			&=\sum\limits_{(k,l,s)\in\mathbb{Z}^{2n+1}}\langle f, e^{-\pi i(l\cdot(-k^{'}) -k\cdot(-l^{'}))}T^{t}_{(k-k^{'},l-l^{'})}\p_{s}\rangle	T^{t}_{(k,l)}\p_{s}\\
			&=\sum\limits_{(k,l,s)\in\mathbb{Z}^{2n+1}}\langle T^{t}_{(k^{'},l^{'})}f, \t\p_{s}\rangle\t\p_{s}\\
			&=ST^{t}_{(k^{'},l^{'})}f.
	\end{align*} 
	Thus $T^{t}_{(k^{'},l^{'})}S =ST^{t}_{(k^{'},l^{'})}.$ Let $R$ be the range operator associated with $S$ satisfying $(\j \circ S)f(\c,\d)=R(\c,\d)(\j f(\c,\d))$. Now, \[\langle Sf, f\rangle=\I\I \langle \j Sf(\c,\d), \j f(\c,\d)\rangle\, d\d d\c=\I\I \langle R(\c,\d)(\j f(\c,\d)), \j f(\c,\d)\rangle\, d\d d\c.\numberthis\label{sf1}\]
	On the other hand, by using Lemma \ref{Hlemma}, we get
	\begin{align*}\label{sf2}
		\langle Sf, f\rangle&=\sum\limits_{(k,l,s)\in\mathbb{Z}^{2n+1}}|\langle f, \t\p_{s}\rangle|^{2}\\
		&=\I\I \sum\limits_{s\in\z}|\langle \j f(\c,\d), \j \p_{s}(\c,\d)\rangle|^{2}\,d\d d\c\\
		&=\I\I \langle \widetilde{G}(\c,\d)(\j f(\c,\d)), \j f(\c,\d)\rangle \, d\d d\c.\numberthis
	\end{align*}
	Define a range operator $\widetilde{R}$ by $\widetilde{R}(\c,\d)=R(\c,\d)-\widetilde{G}(\c,\d).$ Notice that it is a self adjoint operator. By Theorem \ref{characterizationtheorem} and Theorem \ref{Uproperties}, there exists a self adjoint twisted shift preserving operator $W$ such that \[(\j \circ W)f(\c,\d)=\widetilde{R}(\c,\d)(\j f(\c,\d))\numberthis\label{Wrtilda}.\] By combining \eqref{sf1} and \eqref{sf2}, we get 
	\begin{align*}
		0&=\I\I \langle \widetilde{R}(\c,\d)(\j f(\c,\d)), \j f(\c,\d)\rangle\, d\d d\c\\
		&=\I\I \langle (\j\circ W)f(\c,\d), \j f(\c,\d)\rangle\, d\d d\c\\
		&=\langle \j (W f), \j f\rangle=\langle W f,f\rangle 
	\end{align*} 
for all $f\in \v(\A)$, which implies that $W=0$. Hence the corresponding range operator $\widetilde{R}=0$, proving our assertion.
\end{proof}

\begin{theorem}
	Suppose $E^{t}(\A)$ is a frame sequence with constants $A$ and $B$. Let $S$ be the frame operator associated with the system $E^{t}(\A)$. Then $S^{-1}$ is a twisted shift preserving operator. The range operator $R$ associated with $S^{-1}$ is given by $R(\c,\d)=\widetilde{G}(\c,\d)^{-1}$, where $\widetilde{G}(\c,\d)$ is the dual Gramian operator corresponding to the system $\{\j\p_{s}(\c,\d) : s\in \Z\}$, for a.e. $\c,\d\in \T$. 
\end{theorem}
\begin{proof}
For $f\in \l$ and $k,l\in \Z$, we notice that $S^{-1}\t f=S^{-1}\t SS^{-1}f=S^{-1}S\t S^{-1}f$, which in turn implies that $S^{-1}\t=\t S^{-1}$. By Theorem \ref{characterizationtheorem}, there exists a range operator $R$ such that $(\j \circ S^{-1})f(\c,\d)=R(\c,\d)(\j f(\c,\d))$ for all $f\in \v(\A)$, for a.e. $\c,\d\in \T.$ Now, \begin{align*}
	\widetilde{G}(\c,\d)(R(\c,\d)(\j \p_{s}(\c,\d)))&=\widetilde{G}(\c,\d)(\j \circ S^{-1})\p_{s}(\c,\d)\\
	&=(\j \circ S)(S^{-1}\p_{s})(\c,\d)\\
	&=\j\p_{s}(\c,\d).
\end{align*}
Since $\widetilde{G}(\c,\d)R(\c,\d)$ is a bounded linear operator on $J(\c,\d)$, by density argument we can show that $\widetilde{G}(\c,\d)R(\c,\d)=I_{(\c,\d)}$, where $I_{(\c,\d)}$ is the identity operator on $J(\c,\d)$. Similarly, we can show that $R(\c,\d)\widetilde{G}(\c,\d)=I_{(\c,\d)}$. Therefore $R(\c,\d)=\widetilde{G}(\c,\d)^{-1}$ for a.e. $\c,\d\in \T.$ 
\end{proof}
\vspace{10mm}
STATEMENTS AND DECLARATIONS\\

FUNDING : The author (R .R) thanks NBHM, DAE, India for the research project grant. The author (R .V) thanks University Grants Commission, India for the financial support.\\

COMPETING INTERESTS : The authors have no relevant financial or non-financial interests to disclose.\\

AVAILABILITY OF DATA : Not applicable.\\

CODE AVAILABILITY : Not applicable.

\vspace{5mm}
	
	\bibliographystyle{amsplain}
	\bibliography{shiftpreserving}
	\end{document}